\documentclass[12pt,a4paper]{article}
\usepackage{a4,a4wide}
\usepackage{amsfonts,amsmath,amssymb,amsthm}

\usepackage{color}
\usepackage{ifpdf}
\ifpdf
    \usepackage[pdftex]{graphicx}
    \DeclareGraphicsExtensions{.png,.pdf,.jpg}
\else
   \usepackage[dvips]{graphicx}
   \DeclareGraphicsExtensions{.eps}
\fi
\graphicspath{{.}{figuresPL/}}
\usepackage[latin1]{inputenc}
\numberwithin{equation}{section}

\newtheorem{theorem}{Theorem}[section]
\newtheorem{lemma}[theorem]{Lemma}
\newtheorem{proposition}[theorem]{Proposition}

\newtheorem*{theorem*}{Theorem}

\renewcommand\tilde{\widetilde}

\def\R{\mathbb{R}}

\def\E{\mathbb{E}}

\def\Z{\mathbb{Z}}

\def\LM#1{\hbox{\vrule width.2pt \vbox to#1pt{\vfill \hrule width#1pt
height.2pt}}}
\def\LL{{\mathchoice {\>\LM7\>}{\>\LM7\>}{\,\LM5\,}{\,\LM{3.35}\,}}}
\def\restr{{\LL}}

\def\1{\mathbf{1}}

\def\XXint#1#2#3{{\setbox0=\hbox{$#1{#2#3}{\int}$ }
\vcenter{\hbox{$#2#3$ }}\kern-.57\wd0}}

\def\eps{\varepsilon}
\renewcommand{\subset}{\subseteq}

\def\lt{\left}
\def\rt{\right}
\def\les{\lesssim}
\def\ges{\gtrsim}

\def\Wper{W_{\textrm{per}}}

\def\supp{\textup{spt}\,}

\begin{document}
\title{A variational approach to regularity theory in optimal transportation}
\author{M. Goldman\thanks{ Universit\'e de Paris, CNRS, Sorbonne-Universit\'e,  Laboratoire Jacques-Louis Lions (LJLL), F-75005 Paris, France, \texttt{goldman@math.univ-paris-diderot.fr}}} 
\maketitle
\begin{abstract}\noindent
This paper describes recent results obtained in collaboration with M. Huesmann and  F. Otto on the regularity of optimal transport maps. The main result is a quantitative version
of the well-known fact that the linearization of the Monge-Amp\`ere equation around the identity is the Poisson equation. We present two applications of this result. The first one is 
a variational proof of the partial regularity theorem of Figalli and Kim and the second is the rigorous validation 
of some predictions made by Carraciolo and al. on the structure of the optimal transport maps in matching problems.   
\end{abstract}

\section{Introduction}
Following Caffarelli's groundbreaking papers \cite{CafJAMS92,CafAoM90bis}, the classical approach to regularity theory for solutions of the optimal transport problem goes through
maximum principle arguments and the construction of barriers (see the review paper \cite{DePFig_notes}). The aim of this note is to describe a recent alternative approach, more variational
in nature and based on the fact that the linearization of the Monge-Amp\`ere equation around the identity is the Poisson equation (see \cite{Viltop}). 
Our main achievement in this direction is an harmonic approximation result which says that if at a given scale the transport plan is close to the identity and if at the same scale both the starting and target measures are close
(in the Wasserstein metric) to be constant, then on a slightly smaller scale, the transport plan is actually extremely close to an harmonic gradient field. 
As in De Giorgi's approach to the regularity theory for minimal surfaces (see \cite{Maggi}) this allows to transfer the good regularity properties of
harmonic functions to the transport plan and obtain an ``excess improvement by tilting'' estimate. This may be used to propagate information from the macroscopic scale down 
to the microscopic scale through a Campanato iteration.\\
We give two applications of this result. The first one is a new proof of the partial regularity result of Figalli and Kim \cite{FigKim} (see also \cite{DePFig}). 
The second one is a validation up to the microscopic scale of the prediction by Caracciolo and al. \cite{CaLuPaSi14} that for the optimal matching
problem between a Poisson point process and the Lebesgue measure, the optimal transport plan is well approximated by the gradient of the solution to the corresponding Poisson equation with very high probability. \\

The plan of this note is the following. In Section \ref{sec:OT} we recall some standard results on optimal transportation. The harmonic approximation theorem is stated together with a sketch of proof in Section \ref{sec:harm}.
We then describe the application to the partial regularity result in Section \ref{sec:smallscale} and to the optimal matching problem in Section \ref{sec:largescale}.

\section{The optimal transport problem}\label{sec:OT}
Optimal transportation is nowadays a very broad and active field. We give here only a very basic and short introduction to the topic and refer the reader to the monographs \cite{Santam,Viltop}
for much more details.  For $\mu$ and $\lambda$ two positive measures on $\R^d$ with $\mu(\R^d)=\lambda(\R^d)$ the optimal transport problem (in its Lagrangian formulation) is 
\begin{equation}\label{OTLag}
 W^2(\mu,\lambda)=\inf_{\pi_1=\mu,\pi_2=\lambda} \int_{\R^d\times \R^d}|x-y|^2d\pi,
\end{equation}
where for a coupling $\pi$ on $\R^{d}\times \R^d$, $\pi_1$ (respectively $\pi_2$) denotes the first (respectively the second) marginal of $\pi$. Under very mild assumptions on $\mu$ and $\lambda$
(for instance compact supports), an optimal transference plan $\pi$ exists (see \cite{Viltop}). The optimality conditions are as follows:
\begin{theorem}\label{theo:brenier}
 Let $\pi$ be a coupling between $\mu$ and $\lambda$. 
 \begin{itemize}
 \item[(i) ] (Knott-Smith) It is optimal if and only if there exists a convex and 
 lower-semicontinuous function $\psi$ (also called the Kantorovich potential) such that $\supp \pi\subset \textrm{Graph}(\partial \psi)$. 
 \item[(ii)] (Brenier) Moreover, if $\mu$ does not give mass to Lebesgue negligible sets, then there exists a unique  $\nabla \psi$, gradient of a convex function, with $\nabla \psi\# \mu=\lambda$ and $\pi=(\textrm{id}\times\nabla \psi)\#\mu$. In this case we let $T=\nabla \psi$ be the optimal transport map.  
\end{itemize}
 \end{theorem}
Let us point out that assuming that $\psi$ is regular and that both $\mu$ and $\lambda$ are smooth densities, the condition $T\#\mu=\lambda$ is nothing else than the Monge-Amp\`ere equation
\[
  \det \nabla^2 \psi= \frac{\mu}{ \lambda \circ \nabla \psi}.
\]
In particular, if both $\mu$ and $\lambda$ are close to (the same) constant density, then the Monge-Amp\`ere equation linearizes to the Poisson equation (see \cite[Ex. 4.1]{Viltop})
\begin{equation}\label{Poi}
 \Delta \psi=\mu-\lambda.
\end{equation}

We will also use the  Eulerian formulation of the optimal transport  problem.
\begin{theorem}[Benamou-Brenier]
 There holds
 \begin{equation}\label{BB}
   W_2^2(\mu,\lambda)=\inf_{(\rho,j)} \lt\{ \int_{\R^d}\int_0^1 \frac{1}{\rho}|j|^2 \ :  \partial_t\rho+\nabla\cdot j=0, \ \rho_0=\mu, \ \rho_1=\lambda\rt\}.
 \end{equation}
Moreover, if $\pi$ is an optimal transport plan for \eqref{OTLag}, then the density-flux pair $(\rho_t,j_t)$ defined for $t\in[0,1]$ by its action on test functions $(\zeta,\xi)\in C^0(\R^d)\times (C^0(\R^d))^d$ as
\begin{equation}\label{defj}
 \int_{\R^d} \zeta d\rho_t=\int_{\R^d\times\R^d} \zeta( (1-t)x+ty) d\pi \quad \textrm{ and } \quad \int_{\R^d} \xi\cdot dj_t=\int_{\R^d\times\R^d} \xi( (1-t)x+ty)\cdot(y-x) d\pi,
\end{equation}
is a minimizer of \eqref{BB}.
\end{theorem}
Let us introduce some further notation. If $(\rho,j)$ is a minimizer of \eqref{BB}, we define $(\bar \rho, \bar j)$ the density-flux pair 
obtained by integrating in time (for instance $\bar \rho=\int_0^1 \rho_t$). For $R>0$ and $\mu$ a positive measure on $\R^d$, we denote by $W_{B_R}(\mu,\kappa)=W(\mu\restr B_R, \kappa \chi_{B_R} dx)$, the Wasserstein distance between the restriction of $\mu$ to the ball $B_R$ and the corresponding constant density $\kappa= \frac{\mu(B_R)}{|B_R|}$. \\
In order to obtain a local version of the equivalence between \eqref{OTLag} and \eqref{BB}, we will need an $L^\infty$ bound on the displacement (see \cite{GO,GHO1}).
\begin{lemma}\label{Linf}
  Let $\pi$ be a coupling between two measures $\mu$ and $\lambda$. Assume that $\supp \pi$ is monotone\footnote{Meaning that for every $(x_1,y_1)$ and $(x_2,y_2)$ in $\supp \pi$, $(x_1-x_2)\cdot(y_1-y_2)\ge 0$.} and that for some\footnote{We use the short-hand notation 
  $A\ll 1$ to indicate that there exists $\eps>0$ depending only on the dimension  
 such that $A\le \eps$. Similarly, $A\les B$ means that there exists a 
constant $C>0$ depending on the dimension such that $A\le C B$.} $R>0$, $E+D\ll1$ where 
  \begin{equation}\label{E}
   E= \frac{1}{R^{d+2}}\int_{(B_{6R}\times \R^d)\cup(\R^d\times B_{6R})} |x-y|^2 d\pi 
   \end{equation}
   and
   \begin{equation}\label{D}
   D = \frac{1}{R^{d+2}}W^2_{B_{6R}}(\mu,\kappa_\mu)+ \frac{1}{\kappa_\mu}(\kappa_\mu-1)^2+\frac{1}{R^{d+2}} W^2_{B_{6R}}(\lambda, \kappa_\lambda)+ \frac{1}{\kappa_\lambda}(\kappa_\lambda-1)^2.
   \end{equation}
Then, for every $(x,y)\in \supp \pi \cap \lt((B_{5R}\times \R^d)\cup (\R^d\times B_{5R})\rt)$ 
\begin{equation}\label{wg31}
 |x-y|\les R \lt(E+D\rt)^{\frac{1}{d+2}}.
\end{equation}
\end{lemma}

\section{The harmonic approximation theorem}\label{sec:harm}
We now state the harmonic approximation theorem. By scaling invariance, it is enough to state it at the unit scale $R=1$. For $\mu$, $\lambda$ two positive measures and $\pi$ an optimal coupling between them, we define the ``excess'' energy $E$ as in \eqref{E} and the 
 distance to the data $D$ as in \eqref{D}.

\begin{theorem}(\cite[Th. 1.4]{GHO1})\label{theo:harmonicLagintro}
 For every $0<\tau\ll 1$, there exist $\eps(\tau)>0$ and $C(\tau)>0$ such that provided $E+D\le \eps$,  there exists a radius $R\in (3,4)$ such that if $\Phi$ is a  solution of ($\nu$ denotes here the external normal to $\partial B_R$)
 \begin{equation}\label{ma89}
  \Delta \Phi=c \ \textrm{ in } B_{R} \qquad \textrm{ and } \qquad \nu \cdot \nabla \Phi=\nu\cdot \bar j \ \textrm{ on } \partial B_{R},
 \end{equation}
where  $c$ is the generic constant for which this equation is solvable, then
\begin{equation}\label{eq:mainestimateintro}
 \int_{(B_{1}\times \R^d)\cup(\R^d\times B_{1})} |x-y+\nabla \Phi(x)|^2 d\pi\le \tau E+ C D.
\end{equation}
\end{theorem}
The proof of Theorem \ref{theo:harmonicLagintro} is actually performed at the Eulerian level. Thanks to Lemma \ref{Linf}, it is indeed enough to prove:
\begin{theorem}\label{theo:harmonicEul}
 For every $0<\tau\ll 1$, there exist $\eps(\tau)>0$ and $C(\tau)>0$ such that provided $E+D\le \eps$,  there exists a radius $R\in (3,4)$ such that if $\Phi$ solves \eqref{ma89}, then
\begin{equation}\label{eq:maineulerian}
\int_{B_2}\int_0^1 \frac{1}{\rho}|j-\rho\nabla \Phi|^2 \le \tau E+ C D.
\end{equation}
\end{theorem}
To simplify a bit the discussion, we will assume from now on that $\lambda=\kappa_\mu=1$, so that $D= W^2_{B_{6}}(\mu, 1)$.\\
The proof of Theorem \ref{theo:harmonicEul} is based on three ingredients. 
The first of them is the  choice of  a 'good' radius $R$. Indeed, as will become apparent in the discussion below,  we need a control on various quantities and this seems to be possible only for generic radii.
The second ingredient is an almost orthogonality property. The last one is  the construction of a competitor for \eqref{BB}. \\
We define the measure $f= \nu\cdot j$ on $\partial B_R\times (0,1)$ and then let $\bar f=\int_0^1 f= \nu\cdot \bar j$. Before discussing the almost
orthogonality property and the construction, let us point out that for our estimates we would need to control the Dirichlet energy
$\int_{B_R}|\nabla \Phi|^2$ by $E+D$. Since by elliptic regularity,
\begin{equation}\label{elliptic}
 \int_{B_R}|\nabla \Phi|^2\les \int_{\partial B_R} {\bar f}^2,
\end{equation}
this is only possible if $\bar f$ is controlled in $L^2$ (or at least in $H^{-1/2}$). 
In order to solve this issue,  \eqref{eq:maineulerian} is first proven with $\phi$ instead of $\Phi$ where $\phi$ solves
\[
 \Delta \phi=c \ \textrm{ in } B_{R} \qquad \textrm{ and } \qquad \nu \cdot \nabla \phi=\hat{g} \ \textrm{ on } \partial B_{R},
\]
where $\hat{g}$ is  a regularized version of $\bar f$ in the sense that\footnote{For a measure $\mu$, we note $\mu_{\pm}$ its positive/negative part.} 
\begin{equation}\label{boundg}
 \int_{\partial B_R} \hat{g}^2\les E+D \qquad \textrm{and} \qquad W^2(\bar f_{\pm},\hat{g}_{\pm})\les \lt(E+D\rt)^{\frac{d+3}{d+2}}.
\end{equation}
 The density $\hat{g}$ is obtained by projection on $\partial B_R$, using the fact that for 'good' radii, thanks to \eqref{wg31}, 
 the number of particles crossing $\partial B_R$ is controlled by $E+D$. We will however forget here about this difficulty and assume that we may choose $\hat{g}=\bar f$ (and thus $\phi=\Phi$). In particular, in view of \eqref{boundg},
 we will assume that we have the bound 
 \begin{equation}\label{boundf}
  \int_{\partial B_R}\int_0^1 f^2\les E+D.
 \end{equation}

We may now state the almost-orthogonality property: 
 \begin{lemma}(Orthogonality)\label{Lort}
  For every $0<\tau\ll 1$, there exist $\eps(\tau)>0$ and $C(\tau)>0$ such that if $E+D\le \eps$, 
  \begin{equation}\label{io15intro}
   \int_{B_2}\int_0^1 \frac{1}{\rho}|j-\rho\nabla \Phi|^2
   \le  \lt(\int_{B_R}\int_0^1 \frac{1}{\rho}|j|^2-\int_{B_R}|\nabla \Phi|^2\rt)+\tau E +C D.
  \end{equation}
 \end{lemma}

\begin{proof}[Sketch of proof]
 Expanding the squares we have 
 \begin{align*}
   \int_{B_R}\int_0^1 \frac{1}{\rho}|j-\rho\nabla \Phi|^2&=\lt(\int_{B_R}\int_0^1 \frac{1}{\rho}|j|^2-\int_{B_R}|\nabla \Phi|^2\rt)\\
   &\qquad +2\int_{B_R}\int_0^1 (\nabla \Phi-j)\cdot \nabla \Phi +\int_{B_R}(\bar \rho-1)|\nabla \Phi|^2.
 \end{align*}
Let us estimate the two error terms. Using integration by parts we have  (assuming without loss of generality that $\int_{B_R} \Phi=0$)
\begin{align*}
 \int_{B_R}\int_0^1 (\nabla \Phi-j)\cdot \nabla \Phi &= - \int_{B_R}\int_0^1 (\Delta \Phi- \nabla \cdot j) \Phi  +\int_{\partial B_R} \int_0^1 (\bar f- f) \Phi\\
 &=-\int_{B_R}\int_0^1 \partial_t \rho \Phi\\
 &=\int_{B_R} \Phi d(\mu-1)
\end{align*}
Forgetting higher order terms (and assuming that $\frac{\mu(B_R)}{|B_R|}=1$), we have (recall that the Wasserstein distance is homogeneous to the $H^{-1}$ norm)
\begin{equation}\label{eq:orth1}
 \lt|\int_{B_R} \Phi d(\mu-1)\rt|\les \lt(\int_{B_R} |\nabla \Phi|^2\rt)^{1/2} W_{B_R}(\mu,1)\stackrel{\eqref{elliptic}\&\eqref{boundf}}{\les} (E+D)^{1/2} D^{1/2}\stackrel{\textrm{Young}}{\le} \tau E+\frac{C}{\tau} D.
\end{equation}
Regarding the second term, in the case when $\mu=\chi_{\Omega}$ for some set $B_6\subset \Omega$, we may argue as in \cite[Lem. 3.2]{GO} and obtain that by McCann's displacement convexity, 
$\bar \rho\le 1$ and thus $\int_{B_R}(\bar \rho-1)|\nabla \Phi|^2\le 0$. For generic measures $\mu$ the argument is more subtle and requires a combination
of elliptic estimates for (a regularized version of) $\Phi$ together with the bound 
\[
 W^2_{B_R}(\bar \rho,1)\les E+ D,
\]
which holds for 'good' radii.

\end{proof}

As explained above, the last ingredient is the construction of a competitor:
\begin{lemma}\label{Lconintro}
 For every $0<\tau\ll1$, there exist $\eps(\tau)>0$ and $C(\tau)>0$ such that if $E+D\le \eps$,  there exists a density-flux pair $(\tilde \rho, \tilde j)$ such that 
 \begin{equation}\label{conteqloc}
  \begin{cases}
   \partial_t \tilde{\rho}+\nabla\cdot \tilde j=0 & \textrm{in } B_R\times (0,1)\\
   \tilde{\rho}_0=\mu, \ \tilde{\rho}_1=1 & \textrm{in } B_R\\
   \nu\cdot \tilde{j}= f &\textrm{on } \partial B_R\times(0,1) 
  \end{cases}
 \end{equation}
and
 \begin{equation}\label{io31intro}
  \int_{B_R}\int_0^1 \frac{1}{\tilde \rho}|\tilde{j}|^2-\int_{B_R}|\nabla \Phi|^2 \le \tau E +C D.
 \end{equation}
\end{lemma}
\begin{proof}[Sketch of proof]
 We may assume for simplicity that also $\mu=1$ in $B_R$. Indeed, otherwise we can  connect in the time interval $(0,\tau)$, the measure $\mu$  (in $B_R$) 
 to the constant density $1$ at a cost of order $\frac{1}{\tau} W_{B_R}^2(\mu,1)=\frac{1}{\tau}D$.\\  
Let $0<r\ll1$  be a small parameter to be chosen later on. We make the construction separately in the bulk $B_{R-r}\times(0,1)$ and in the boundary layer $B_R\backslash B_{R-r}\times (0,1)$ and set 
\[
 \tilde{\rho}=\begin{cases}
               1 & \textrm{in } B_{R-r}\times(0,1)\\
               1+s& \textrm{in } B_R\backslash B_{R-r}\times(0,1),
              \end{cases}
  \qquad  \tilde{j}=\begin{cases}
               \nabla \Phi & \textrm{in } B_{R-r}\times(0,1)\\
               \nabla \Phi+q& \textrm{in } B_R\backslash B_{R-r}\times(0,1),
              \end{cases}
\]
and require that $|s|\le 1/2$, $\partial_t s+\nabla \cdot q=0$ in $B_R\backslash B_{R-r}\times(0,1)$, $s_0=s_1=0$ in $B_R$ and  $\nu\cdot q=f-\bar f$ on $\partial B_R\times (0,1)$, so that \eqref{conteqloc} is satisfied.
 The existence of an admissible pair $(s,q)$ satisfying the energy bound
 \begin{equation}\label{energsq}
  \int_{B_R\backslash B_{R-r}}\int_0^1 |q|^2\les r\int_{\partial B_R} \int_0^1 (f-\bar f)^2
 \end{equation}
as long as $r\gg \lt(\int_{\partial B_R} \int_0^1( f-\bar f)^2\rt)^{1/(d+1)}$ is obtained arguing by duality, in the same spirit as \cite[Lem.3.3]{ACO} (see \cite[Lem. 2.4]{GO}). \\
We may now estimate
\begin{align*}
 \int_{B_R}\int_0^1 \frac{1}{\tilde \rho}|\tilde{j}|^2-\int_{B_R}|\nabla \Phi|^2& \le \int_{B_R\backslash B_{R-r}}\int_0^1 \frac{1}{1+s}|\nabla \Phi+q|^2\\
 &\les \int_{B_R\backslash B_{R-r}} |\nabla \Phi|^2 +\int_{B_R\backslash B_{R-r}}\int_0^1 |q|^2\\
 &\stackrel{\eqref{energsq}}{\les} r \int_{\partial B_R} \bar f^2 +r \int_{\partial B_R} \int_0^1 (f-\bar f)^2\les r \int_{\partial B_R} f^2,
\end{align*}
where we used that by elliptic regularity, $\int_{B_R\backslash B_{R-r}} |\nabla \Phi|^2\les r \int_{\partial B_R} \bar f^2$. Choosing $r$ to be a large multiple of $\lt(\int_{\partial B_R} \int_0^1( f-\bar f)^2\rt)^{1/(d+1)}$ yields
\[
 \int_{B_R}\int_0^1 \frac{1}{\tilde \rho}|\tilde{j}|^2-\int_{B_R}|\nabla \Phi|^2\les \lt(\int_{\partial B_R}\int_0^1 f^2\rt)^{\frac{d+2}{d+1}}\stackrel{\eqref{boundf}}{\les} (E+D)^{\frac{d+2}{d+1}},
\]
which concludes the proof of \eqref{io31intro} since $\frac{d+2}{d+1}>1$ and $E+D\ll1$.
 \end{proof}
 \begin{proof}[Proof of Theorem \ref{theo:harmonicEul}]
By (local) minimality of $(\rho,j)$, we have $\int_{B_R}\int_0^1 \frac{1}{\rho}|j|^2\le \int_{B_R}\int_0^1 \frac{1}{\tilde\rho}|\tilde j|^2$ so 
that combining \eqref{io15intro} and \eqref{io31intro} together gives the desired estimate \eqref{eq:maineulerian}.
\end{proof}
\section{Application to partial regularity}\label{sec:smallscale}
We now turn to applications of Theorem \ref{theo:harmonicLagintro} and start with a partial regularity result. Here we are interested in the behavior at small scales. \\
Let us first recall the  main regularity result for optimal transport maps due to Caffarelli \cite{CafAoM90bis,CafJAMS92}.
\begin{theorem}\label{theo:Caf}
 If $\mu$ and $\lambda$ have compact supports, are absolutely continuous with respect to the Lebesgue measure with densities bounded from above and below on their support
 and if $\supp \lambda$ is convex, then the optimal transport map $T$ from $\mu$ to $\lambda$ is $C^{0,\alpha}$.
\end{theorem}
The hypothesis that $\supp \lambda$ is convex is not merely technical. Indeed, considering for instance the optimal transport map between one ball and two disjoint balls, 
it is easy to construct examples where the optimal transport map is discontinuous. However, building on the ideas of Caffarelli to prove Theorem \ref{theo:Caf}, 
Figalli and Kim proved in \cite{FigKim} that even without the convexity assumption on $\supp \lambda$,  the singular set of $T$ cannot be too big (see also \cite{DePFig} for a generalization to arbitrary non-degenerate cost functions).
\begin{theorem}\label{theo:FigKim}
  Let  $\mu$ and $\lambda$ be probability measures with compact supports, both  absolutely continuous with respect to the Lebesgue measure with densities bounded from above and below on their support.
 Then, there exist open sets $\Omega\subset \supp \mu$ and $\Omega'\subset \supp \lambda$  with $|\supp \mu\backslash \Omega|=|\supp \lambda\backslash \Omega'|=0$ and such that 
 the optimal transport map $T$ from $\mu$ to $\lambda$ is a $C^{0,\alpha}$ homeomorphism between $\Omega$ and $\Omega'$.
\end{theorem}
Let us point out that it is actually conjectured that the singular set is much smaller and 
has the same structure as the singular set of gradients of convex functions i.e. that it is $n-1$-rectifiable (see \cite{KitMac} for a result in this direction).\\
A first application of Theorem \ref{theo:harmonicLagintro} is a new proof of Theorem \ref{theo:FigKim} (under the additional hypothesis that $\mu$ and $\lambda$ are H\"older continuous). 
For the sake of simplicity, we will assume from now on that $\mu=\chi_{\Omega_1}$ and $\lambda=\chi_{\Omega_2}$ for some bounded open sets $\Omega_i$ (so that in particular with the notation of Section \ref{sec:harm}, $D=0$). 
As in \cite{DePFig}, we derive Theorem \ref{theo:FigKim} combining Alexandrov's Theorem (see \cite{Viltop}), which state that $T$ is differentiable a.e., 
  with an $\eps-$regularity theorem.
  \begin{theorem}(\cite[Th.1.2]{GO})\label{theo:epsreg}
   Let $T$ be the optimal transport map from $\Omega_1$ to $\Omega_2$. For every $\alpha \in (0,1)$, there exists $\eps(\alpha)$ such that if $R>0$ is such that $B_{6R}\subset \Omega_1\cap \Omega_2$ and 
   \[
    \frac{1}{ R^{d+2}}\int_{B_{6R}}|T-x|^2\le \eps,
   \]
then $T\in C^{1,\alpha}(B_R)$.
  \end{theorem}

By scaling invariance, we may assume that $R=1$. As already alluded to the proof goes through a Campanato iteration. Indeed, by Campanato's characterization of $C^{1,\alpha}$ spaces (see  \cite{campanato}), it is enough to prove that for every $0<r\le \frac{1}{6}$,
\[
 \min_{A,b} \frac{1}{r^{d+2}} \int_{B_r} |T-(Ax+b)|^2\les  r^{2\alpha} \int_{B_{1}}|T-x|^2. 
\]
Defining 
\[E(T,R)= \frac{1}{ R^{d+2}}\int_{B_{6R}} |T-x|^2,\]
this is in turn obtained by using iteratively the following proposition.
\begin{proposition}
 For every $\alpha\in(0,1)$, there exist $\theta(\alpha)\in(0,1)$ and $\eps(\alpha)$ such that  if $B_{6R}\subset \Omega_1\cap \Omega_2$  and $E(T,R)\le \eps$, there exist a symmetric matrix $B$ with $\det B=1$ and a vector $b$  such that letting 
 $\hat{T}(x)=B(T(Bx)-b)$, 
 \[
  E(\hat{T},\theta R)\le \theta^{2\alpha} E(T,R)
 \]
and $\hat{T}$ is the optimal transport map between $\hat{\Omega}_1=B^{-1}\Omega_1$ and $\hat{\Omega}_2=B(\Omega_2-b)$.
\end{proposition}
\begin{proof}[Sketch of proof]
By scaling we may assume that $R=1$.\\
Let $\tau\ll \frac{\theta^{2\alpha}}{\theta^{d+2}}$ be fixed. Applying Theorem \ref{theo:harmonicLagintro}, we find the existence of a function $\Phi$ which is harmonic in $B_2$ (under our assumptions $c=0$ in \eqref{ma89}) and such that  (since $D=0$)
\begin{equation}\label{eq:onestep}
 \int_{B_1} |T-(x+\nabla \Phi)|^2\le \tau E(T,1)
\end{equation}
and (recall \eqref{elliptic} and \eqref{boundf})
\begin{equation}\label{eq:phistep}
 \int_{B_2} |\nabla \Phi|^2\les E(T,1).
\end{equation}
We then define $b=\nabla \Phi(0)$ and $B= \exp(-A/2)$ where $A=\nabla^2 \Phi(0)$. Since $\Phi$ is harmonic $\textrm{Tr} A=0$ and thus $\det B=1$. Notice that  if $T=\nabla \psi$ for some convex function $\psi$ (by Theorem \ref{theo:brenier}), then $\hat{T}=\nabla \hat \psi$ with 
$\hat{\psi}(x)=\psi(Bx)- b\cdot x$, which is also a convex function. Therefore $\hat{T}$ is the optimal transport map between $\hat{\Omega}_1$ and $\hat{\Omega}_2$.
 We may now estimate
 \begin{align*}
  E(\hat{T},\theta)&= \frac{1}{\theta^{d+2}}\int_{B_{6\theta}} |\hat{T}-x|^2\\
  &= \frac{1}{\theta^{d+2}}\int_{B^{-1}(B_{6\theta})} |B (T-b)- B^{-1}x|^2\\
  &\les \frac{1}{\theta^{d+2}}\int_{B_{7\theta}}| T-b- B^{-2} x|^2\\
  &\les \frac{1}{\theta^{d+2}}\int_{B_{1}}|T-(x+\nabla \Phi)|^2 +\frac{1}{\theta^{d+2}}\int_{B_{7\theta}} |\nabla \Phi -(Ax+b)|^2\\
  &\qquad +\frac{1}{\theta^{d+2}}\int_{B_{7\theta}} |(\exp A - \textrm{id} -A)x|^2\\
  &\stackrel{\eqref{eq:onestep}}{\les} \frac{\tau}{\theta^{d+2}} E(T,1) + \theta^2\sup_{B_{7\theta}} |\nabla^3 \Phi|^2 + |A|^4\stackrel{\eqref{eq:phistep}}{\les}\frac{\tau}{\theta^{d+2}} E(T,1)+ \theta^2 E(T,1) +E(T,1)^2.
 \end{align*}
This concludes the proof since we chose $\tau\ll \frac{\theta^{2\alpha}}{\theta^{d+2}}$ and since $E(T,1)\ll1$.
\end{proof}

\section{Application to the optimal matching problem}\label{sec:largescale}
We now present an application to the optimal matching problem. As opposed to the previous section, we are interested here at large scales.\\
Over the last thirty years, optimal matching problems have been the subject of intensive work. We refer for instance to the monograph \cite{Ta14}. 
One of the simplest example is the problem of matching the empirical measure of a Poisson point process  to the corresponding Lebesgue measure.
More specifically, we consider for $L\gg1$ a Poisson point process $\mu$ on the the torus $Q_L=[-L/2,L/2)^d\simeq(\R/L\Z)^d$ i.e. 
\[
 \mu=\sum_{i=1}^n \delta_{X_i}
\]
with $X_i$ iid random variables uniformly distributed in $Q_L$ and $n$ a random variable with Poisson distribution with parameter $L^d$. The problem is to estimate the random variable
\[
 \frac{1}{L^d} \Wper^2(\mu, \kappa),
\]
where $\Wper$ indicates the Wasserstein distance on the torus $Q_L$, and to understand the structure of the corresponding optimal transport plans. It is well-known since \cite{AKT84} that\footnote{We use the notation $\log$ for the natural logarithm.} 
\begin{equation}\label{AKT}
 \E\lt[\frac{1}{L^d} \Wper^2(\mu, \kappa)\rt]\sim\begin{cases}
                                              \log L & \textrm{if } d=2\\
                                              1 &\textrm{for } d\ge 3
                                             \end{cases}
\end{equation}
and thus $d=2$ is a critical dimension. Recently, Caracciolo and al. used the  ansatz that the optimal  displacement  should be well approximated by $\nabla \varphi_L$, where $\varphi_L$ solves the Poisson equation (recall \eqref{Poi})
\begin{equation}\label{eqphi}
\Delta \varphi_L=\mu-\kappa \qquad \textrm{in } Q_L
\end{equation}
to make numerous predictions about the optimal prefactor in \eqref{AKT} as well as the correlations (see \cite{CaLuPaSi14,CaSi15}). 
At the macroscopic scale, this ansatz has been partially rigorously justified by Ambrosio and al. (see \cite{AmStTr16,AGS19} and also \cite{ledoux} for a result about the fluctuations) in dimension $2$. To state their result\footnote{The results of \cite{AmStTr16,AGS19}
are stated on the unit cube with a (deterministic) number of points $n\to \infty$. However, their results may be easily transposed into our setting by scaling.}, let us introduce some notation. For $t>0$, denote 
the heat kernel on $Q_L$ by $P_t$ and let $\varphi_{L,t}=P_t\ast \varphi_L$, so that $\varphi_{L,t}$ solves
\[
\Delta \varphi_{L,t}=P_t\ast\mu-\kappa \qquad \textrm{in } Q_L.
\]
\begin{theorem}\label{theo:Ambrosio}
 Let $d=2$, then 
 \begin{equation}\label{estimW2Amb}
  \lim_{L\to \infty} \frac{1}{\log L} \E\lt[\frac{1}{L^2} \Wper^2(\mu, \kappa)\rt]=\frac{1}{2\pi}.
 \end{equation}
Moreover, if $\pi_L$ is the optimal transport plan between $\mu$ and $\kappa$, then setting $t_L=\log^4 L$, for $L\gg1$ there holds
\begin{equation}\label{estimmapAmb}
 \frac{1}{\log L}\E\lt[\frac{1}{L^2}\int_{Q_L\times Q_L} |y-x-\nabla \varphi_{L,t_L}|^2 d\pi_L\rt]\ll \lt(\frac{\log \log L}{\log L}\rt)^{1/2}.
\end{equation}

\end{theorem}
Since by \eqref{estimW2Amb}, the displacement $y-x$ is on average of the order of $\log^{\frac{1}{2}} L$, \eqref{estimmapAmb} shows that 
$\nabla \varphi_{L,t_L}$ indeed coincides with the displacement to leading order. This leaves open the description of the optimal transport plan $\pi_L$ at the microscopic scale.
To state our main result, fix a smooth cut-off function  (which plays a similar role as the heat kernel in Theorem \ref{theo:Ambrosio})
\[
 \eta\in C^\infty_c(B_1) \textrm{ with } \int_{\R^2} \eta=1, \qquad \textrm{and set }\qquad  \eta_R=\frac{1}{R^2}\eta\lt(\frac{\cdot}{R}\rt).
\]
In \cite{GHOinprep}, we prove the following result (see also \cite[Th. 1.2]{GHO1} and \cite[Th. 1.1]{GHO2}):
\begin{theorem}\label{theo:local}
 There exists a stationary random variable $r_*\ge 1$ on $Q_L$ with exponential moments such that if $\bar x\in Q_L$ is such that $r_*(\bar x)\ll L$, then 
 \begin{equation}\label{weakestim}
  \lt|\frac{\int_{Q_L\times Q_L} \eta_R(x-\bar x) (y-x)d\pi}{\int_{Q_L\times Q_L} \eta_R(x-\bar x) d\pi}-\eta_R\ast \nabla \varphi_L(\bar x)\rt|\les \frac{\log R}{R} \qquad \forall L\ges R\ges r_*(\bar x).
 \end{equation}
Moreover, there exists   $R=R(\bar x)\sim r_*(\bar x)$ such that defining the shift $h$ by $h(\bar x)=\frac{1}{|B_R|}\int_{\partial B_R(\bar x)} (x-\bar x)\nu\cdot \nabla \varphi_L$, we have 
\begin{equation}\label{optimh}
 |h(\bar x)|^2\les \log L
\end{equation}
and 
\begin{equation}\label{strongestim}\sup\{ |y-x-h(\bar x)| \ : (x,y)\in \supp \pi\cap (B_{r_*}(\bar x)\times \R^2)\}\les  r_*(\bar x) \lt(\frac{ \log r_*(\bar x)}{r_*^2(\bar x)}\rt)^{1/4}.\end{equation}  
\end{theorem}
With respect to \eqref{estimmapAmb}, \eqref{strongestim} proves that (circular) averages of $\nabla \varphi_L$ coincide with the displacement $y-x$ up to an error which is of order one.
Moreover, \eqref{weakestim} shows that after averaging, the displacement is actually extremely close to  averages of $\nabla \varphi_L$ (notice that the error term $\log R/R$ improves as $R$ increases).\\

By stationarity, it is enough to prove Theorem \ref{theo:local} for $\bar x=0$. 
The proof  is based on the following deterministic result (which is a small post-processing of \cite[Th. 1.2]{GHO1}):
\begin{theorem}\label{theo:deter}
 Let $\mu$ be a measure on $Q_L$. If for some  $L\gg r\gg1$, 
 \begin{equation}\label{hypdata}
  \frac{1}{R^2} W_{B_R}^2(\mu,\kappa)\les \log R \qquad \textrm{ for all dyadic } L\ges R\ges r,
 \end{equation}
then 
\begin{equation}\label{weakestimdeter}
  \lt|\frac{\int_{Q_L\times Q_L} \eta_R(x) (y-x)d\pi}{\int_{Q_L\times Q_L} \eta_R(x) d\pi}-\int_{Q_L} \eta_R \nabla \varphi_L\rt|\les \frac{\log R}{R} \qquad \forall L\ges R\ges r.
 \end{equation}
Moreover, there exists   $R\sim r$ such that letting $h=\frac{1}{|B_R|}\int_{\partial B_R} x\nu\cdot \nabla \varphi_L$, we have 
\begin{equation}\label{strongestimdeter}
\frac{1}{r^2}\int_{(B_{r}\times\R^2)\cup(\R^2\times B_{r}(h))} |y-x-h|^2 d\pi\les \log r.
\end{equation}  
\end{theorem}
Notice that \eqref{strongestim} follows from \eqref{strongestimdeter} and the $L^\infty$ bound \eqref{wg31} of Lemma \ref{Linf}. In order to obtain Theorem \ref{theo:local}, Theorem \ref{theo:deter} is combined with a stochastic argument based on \eqref{estimW2Amb} and a concentration-of-measure argument which ensures that \eqref{hypdata} is satisfied for the Poisson point process $\mu$. \\
The main ingredient for the proof of Theorem \ref{theo:deter} is a Campanato iteration scheme similar to
the one leading to Theorem \ref{theo:epsreg} (and mainly based on Theorem \ref{theo:harmonicLagintro}) which allows to transfer the information that \eqref{strongestimdeter}
holds at scale $L$ by \eqref{hypdata} down to the microscopic scale $r$. This is inspired by the approach developed by Armstrong and Smart   in \cite{MR3481355} (and further refined in 
\cite{arXiv:1409.2678}, see also \cite{AKMbook}) for quantitative stochastic homogenization. The main ideas of \cite{MR3481355} take roots themselves in previous works of Avellaneda and Lin (see \cite{MR0910954}) on periodic homogenization. The outcome of the Campanato scheme may be stated as follows (see \cite[Prop. 1.9]{GHO1})
\begin{proposition}
 There exists a sequence of approximately geometric radii $R_k$ i.e. $L\ge R_0\ge \cdots\ge R_K\ges 1$ with $R_{k-1}\ge 2R_k\ges R_{k-1}$, $R_0\sim L$ and $R_K\sim r$ such that defining recursively  the couplings $\pi_k$ by $\pi_0=\pi$ and 
 \[
  \pi_k=(\textrm{id},\textrm{id}-\nabla \Phi_{k-1}(0))\#\pi_{k-1}
 \]
where $\Phi_k$ solves
\[
  \Delta \Phi_k=c \ \textrm{ in } B_{R_k} \qquad \textrm{ and } \qquad \nu \cdot \nabla \Phi_k=\nu\cdot \bar j_k \ \textrm{ on } \partial B_{R_k}
\]
with $j_k$ defined as in \eqref{defj} with $\pi_{k}$ playing the role of $\pi$, we have for $k\in[0,K]$,
\begin{equation}\label{estimEk}
 \frac{1}{R_k^2}\int_{(B_{6R_k}\times \R^2)\cup(\R^2\times B_{6R_k})} |x-y|^2 d\pi_k\les \log R_k
\end{equation}
and 
\begin{equation}\label{estimphik}
 |\nabla \Phi_k(0)|^2\les \log R_k.
\end{equation}

\end{proposition}
Let us point out that by invariance of the Lebesgue measure under translations, $\pi_k$ is the optimal transport plan between $\mu$ and the Lebesgue measure for every $k$ (this is the reason why we make the translation in the target space). \\
Letting $\tilde{h}=\sum_{k=0}^{K-1} \nabla \Phi_k(0)$ and undoing the iterative definition of $\pi_k$, we see that \eqref{estimEk} directly leads to
\eqref{strongestimdeter} with $h$ replaced by $\tilde{h}$. The proof of \eqref{strongestimdeter} is concluded by the estimate (see \cite[Prop. 1.10]{GHO1})
\[
 |h-\tilde{h}|\les \frac{\log r}{r}.
\]
This estimate is also crucial for the proof of \eqref{weakestimdeter}. Let us point out that  a naive estimate using \eqref{estimphik} leads to 
\[|\tilde{h}|^2\les \log^3 L\]
which is suboptimal. In order to obtain a shift with the optimal estimate \eqref{optimh} it is therefore 
important  to take into account cancellations and replace $\tilde{h}$ by $h$.   \\

Let us close this note by pointing out that in dimension $d\ge 3$, the optimal transport plans corresponding to a very closely related optimal matching problem, have been used in 
\cite{HuSt13}
to construct in the limit $L\to \infty$, a stationary and locally optimal coupling between the Poisson point process on $\R^d$ and the Lebesgue measure. 
For $d=2$, such a coupling is expected not to exist. However,  using \eqref{strongestim}  and passing to the limit $L\to \infty$, it is possible to construct (at least in the sense of Young measures) 
a  coupling between the Poisson point process on $\R^2$ and the Lebesgue measure, which is locally optimal  and has stationary  {\it increments}  
(see \cite[Th.1.2]{GHO2}). 

\subsection*{Acknowledgements}
This research has been partially supported by the ANR project SHAPO.

\bibliographystyle{amsplain}
\bibliography{OT}
 \end{document}